\documentclass{article}
\usepackage[all]{xy}
\xyoption{line}
\usepackage{graphicx}

\usepackage{amsmath, amsfonts, amsthm}
\DeclareMathOperator{\Coin}{Coin}

\newcommand{\MC}{\text{\textit{MC}}}

\newtheorem{thm}{Theorem}
\newtheorem{lem}[thm]{Lemma}

\theoremstyle{definition}

\begin{document}

\bibliographystyle{hplain}

\title{Maps on graphs can be deformed to be coincidence free}
\author{P. Christopher Staecker\thanks{
Address: Department of Mathematics and Computer Science, Fairfield
University, Fairfield CT, USA}
\thanks{Email: cstaecker@fairfield.edu}
}

\maketitle
\begin{abstract}
We give a construction to remove coincidence points of continuous maps
on graphs ($1$-complexes) by changing the maps by homotopies. When the
codomain is not homeomorphic to the circle, we show that any pair of
maps can be changed by homotopies to be coincidence free. 
This means that there can
be no nontrivial coincidence index, Nielsen coincidence number, or coincidence
Reidemeister trace in this setting, and the results of our
previous paper ``A formula for the coincidence Reidemeister trace of
selfmaps on bouquets of circles" are invalid.
\end{abstract}

\section{Introduction}
Let $X$ and $Y$ be graphs, which are always assumed to be
nontrivial. Throughout, we will 
consider continuous maps $f,g:X \to Y$ (continuous maps of
$X$ and $Y$ as dimension $1$ CW-complexes)
and examine the coincidence set  
\[ \Coin(f,g) = \{ x \mid f(x) = g(x) \}. \]

The paper \cite{stae09a} attempts, in the special case of bouquets of
circles, to study coincidence points of $f$ and $g$ by computing the
Reidemeister trace, which would then allow the computation of the
Nielsen number of the pair $(f,g)$. This Nielsen number would be a
lower bound on the minimal number of coincidence points achievable by
deforming $f$ and $g$.

A serious error in \cite{stae09a} renders the approach fundamentally
misguided. The approach makes heavy use of the coincidence index,
which is not well-behaved for bouquets of circles. Our main result
(Theorem \ref{mccor}) is that maps $f,g:X \to Y$ of graphs
with $Y$ not homeomorphic to the circle can always be changed by
homotopy to be coincidence free.
Thus any coincidence index in this setting must always  
be zero, and so any Nielsen number or Reidemeister
trace which were being computed in \cite{stae09a} must have the
value zero.  

In Section \ref{intsection} we give our main result. We conclude in
Section \ref{errorsection} with a note on 
the specific errors in \cite{stae09a}.

We would like to thank Robert F. Brown for many helpful suggestions on
the organization of the paper, and the referee for suggestions which
substantially simplified the paper.

\section{Removing coincidences by homotopy}\label{intsection}

Our strategy for removing coincidences can be intuitively described
using a road traffic analogy. Consider a coincidence point which
occurs on the interior of an edge of the 
domain space. Then we parameterize this edge ($1$-cell) as the time interval
$[0,1]$, and we can view the maps $f$ and $g$ as being represented
by a pair of points which travel around the space $Y$. 

Let us imagine
that these points represent cars traveling on a network of single-lane
roads (so that the cars may not pass one another), and a coincidence
point of the maps will represent a collision of the cars.
A removal of a coincidence point by a homotopy will consist of a
strategy for letting the two cars pass one another without colliding.

Avoiding a collision is possible provided that there is a fork in the
network of roads where at least three roads meet: When two cars are
about to collide, one of them reverses direction until the fork is
reached. At this point, the car which reversed direction moves onto
the third road and allows the other to pass. The cars can now proceed
back to their original meeting point, this time with their positions
reversed. Repeating this process before each imminent collision allows
the cars to complete their trips without colliding.

This strategy is formalized as follows:
\begin{thm}\label{remcoin}
Let $f,g:X \to Y$ be maps of connected graphs with $Y$ not a manifold,
and let $x\in \Coin(f,g)$ be a coincidence point in the interior of
some edge. Then there is an arbitrarily small neighborhood $U$ of
$x$ on which $f$ and $g$ can be changed by homotopy to be coincidence free.
\end{thm}
\begin{proof}
Let $\sigma$ be the edge ($1$-cell) containing $x$, which we identify with its
attaching map $\sigma:[0,1]\to X$. Let $x = \sigma(t_0)$, and let $U =
\sigma([t_0-\epsilon,t_0+\epsilon])$ for some small $\epsilon >0$. We
may assume that $f(x)=g(x)$ is a point on the interior of some
$1$-cell $\rho:[0,1]\to Y$. We can parameterize $\sigma$ and $\rho$
so that $f(x)=g(x)=\rho(1/2)$, and that $f$ and $g$ behave according
to the graph in Figure \ref{intgraph}. (We may perhaps have to
interchange the roles of $f$ and $g$.)

\begin{figure}
\newcommand{\xwidth}{60}
\[ \begin{xy}
(0,0)="ll";
(0,\xwidth) = "ul"; 
(\xwidth,0) = "lr";
(\xwidth,\xwidth) = "ur";
"ul";"ll"; **@{-};
"lr";"ll";**@{-};
(10,1);(10,-1);**@{-};
(30,1);(30,-1);**@{-};
(50,1);(50,-1);**@{-};
(-1,0);(0,0);**@{-};
(-1,10);(1,10);**@{-};
(-1,30);(1,30);**@{-};
(-1,50);(1,50);**@{-};
(-1,60);(1,60);**@{-};
(10,-3)*{\sigma(t_0-\epsilon)};
(30,-3)*{\sigma(t_0)};
(50,-3)*{\sigma(t_0+\epsilon)};
(-5,0)*{\rho(0)};
(-5,10)*{\rho(\frac14)};
(-5,30)*{\rho(\frac12)};
(-5,50)*{\rho(\frac34)};
(-5,60)*{\rho(1)};
(65,0)*{\sigma(t)};
(0,65)*{\rho(s)};
(10,50);(50,10); **@{-}; ?(0)*+!L{f};
(10,10);(50,50); **@{-}; ?(0)*+!L{g};
(30,30)*{\circ}
\end{xy} \]
\caption{Behavior of $f$ and $g$ on $U$.\label{intgraph}}
\end{figure}
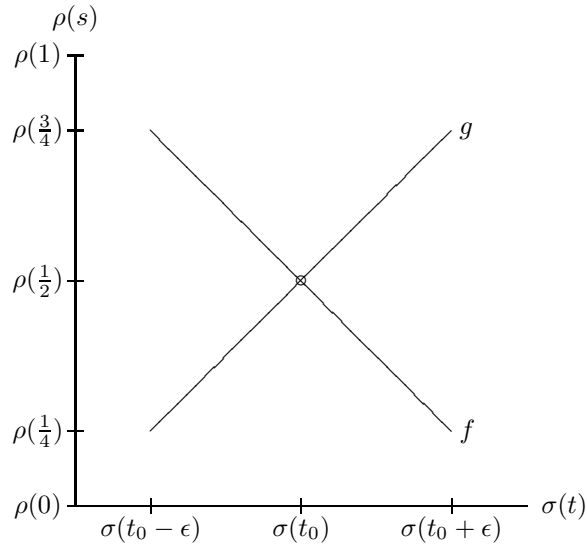

The assumption that $Y$ is not a manifold means that we may choose the
CW-complex structure on $Y$ so that the vertex $\sigma(0)$ meets two
other $1$-cells $\gamma,\lambda:[0,1]\to Y$ with
$\gamma(0)=\lambda(0)=\sigma(0)$. This vertex is the ``fork in the road''.
Now we change $f$ and $g$ by homotopy on $U$ to maps $f'$ and $g'$
according to Figure \ref{intgraphremoved}. Informally, the two maps
retreat to the fork point, use the fork to pass one another
without colliding, and return to their original positions at time $t_0+\epsilon$.
\begin{figure}
\newcommand{\xwidth}{60}
\[ \begin{xy}
(0,0)="ll";
(0,\xwidth) = "ul"; 
(\xwidth,0) = "lr";
(\xwidth,\xwidth) = "ur";
"ul";"ll"; **@{-};
"lr";"ll";**@{-};
(10,1);(10,-1);**@{-};
(30,1);(30,-1);**@{-};
(50,1);(50,-1);**@{-};
(-1,0);(0,0);**@{-};
(-1,10);(1,10);**@{-};
(-1,30);(1,30);**@{-};
(-1,50);(1,50);**@{-};
(-1,60);(1,60);**@{-};
(10,-3)*{\sigma(t_0-\epsilon)};
(30,-3)*{\sigma(t_0)};
(50,-3)*{\sigma(t_0+\epsilon)};
(-3,0)*{0};
(-3,10)*{\frac14};
(-3,30)*{\frac12};
(-3,50)*{\frac34};
(-3,60)*{1};
(65,0)*{\sigma(t)};
(0,65)*{s};
(50,50);(36,0);**@{-}; ?(1)*+!L{g'};
(36,0);(24,25);**@{--};
(24,25);(16,0);**@{--}; 
(16,0);(10,10);**@{-};
(50,10);(44,0);**@{-}; ?(1)*+!L{f'};
(44,0);(36,25);**@{.};
(36,25);(24,0);**@{.};
(24,0);(10,50);**@{-};
\end{xy} \]
\caption{Behavior of $f'$ and $g'$ on $U$. Solid line indicates values
  in $\rho(s)$, dotted line indicates values in $\gamma(s)$, and dashed
  line indicates values in $\lambda(s)$.\label{intgraphremoved}}
\end{figure}
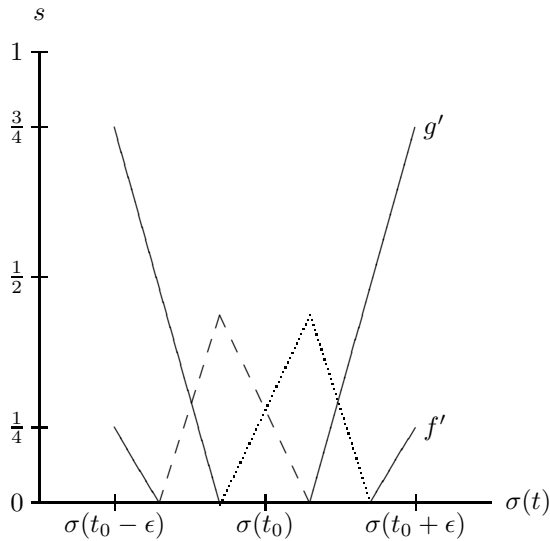
The maps $f'$ and $g'$ are free of coincidences on $U$, and the theorem is proved.
\end{proof}

The above theorem implies our main result, with the help of a lemma which is true for much more general spaces, though we only require it for complexes. Its proof is an exercise.
\begin{lem}
Let $f,g:X \to Y$ where $X$ and $Y$ are connected complexes, and let $x\in \Coin(f,g)$. Then for any neighborhood $U$ of $x$, we may change $f$ and $g$ by homotopy on $U$ so that $x$ is no longer a coincidence point.
\end{lem}

The lemma above means that we may assume that every coincidence of our maps occurs on the interior of an edge, and then Theorem \ref{remcoin} can be applied repeatedly to remove them. Thus we obtain:

\begin{thm}\label{mccor}
If $f,g:X \to Y$ are maps on connected graphs with $Y$ not homeomorphic to
the circle, then $f$ and $g$ can be changed by homotopy to be
coincidence free. 
\end{thm}
\begin{proof}
If $Y$ is homeomorphic to the interval $[0,1]$, then $f$ and $g$ are trivially nullhomotopic and thus can be made to be coincidence free by deforming them into different constant maps. Thus we may assume that $Y$ is not a manifold, and we may freely use Theorem \ref{remcoin}.

First, we may change our maps by homotopy to be ``linear'' as in
\cite{stae09a} so that $\Coin(f,g)$ is a finite set. Furthermore by
the lemma we may assume that all coincidences occur at interior points
of edges. Then repeated application of Theorem \ref{remcoin} will
remove all coincidences. 
\end{proof}

See the end of Section \ref{errorsection} for a note on the case where
$Y$ is the circle.

The above theorem highlights the fact that coincidence theory on
graphs is not a generalization of fixed point theory. It is
certainly possible for a selfmap $f$ on e.g.\ a bouquet of 2 circles to have
fixed points which cannot be removed by homotopy, even though (by
Theorem \ref{mccor}) any
coincidences of $f$ with the identity map \emph{can} be removed. This 
occurs because our removal construction changes the
second map by homotopy as well as the first. This distinction does
not occur between fixed point and coincidence theory on manifolds and
some other spaces, as demonstrated by Brooks in \cite{broo71}, but
Brooks' result does not apply to complexes in general. 

\section{The error of \cite{stae09a}, and the coincidence index}\label{errorsection}
The formula given for the Reidemeister trace in \cite{stae09a} uses essentially two ingredients: the computation of the Reidemeister class for each coincidence point, and the computation of the coincidence index for each coincidence point. The material concerning the Reidemeister class is essentially correct, and the material concerning the index is incorrect.

The error specifically arises on page 43 of \cite{stae09a}: ``Near any
point $x$ other than $x_0$, the space $X$ is an orientable
differentiable manifold, and we define the coincidence index as usual
for that setting.'' This formulation of the coincidence index is not
well-behaved under homotopy. If, over the course of the homotopy, the
coincidence value (the common value of $f(x)$ and $g(x)$) travels
through the wedge point $y_0$, this ``index'' will change unpredictably.

In fact, two fundamental properties of the coincidence index
are that it is invariant under
homotopies of $f$ and $g$, and that the index is zero when $f$ and $g$
are coincidence-free on $U$. Since (by Theorem \ref{mccor}) the
coincidence set for maps of graphs can always be made empty by
homotopies, any ``coincidence index'' in this setting must always give
the value zero. 

Some nontrivial indices can be defined by restricting the
structure of either the domain or the codomain spaces. Gon\c{c}alves
in \cite{gonc99} gives a coincidence index for maps from a complex into a
manifold of the same dimension, which suffices to address the
exceptional case from Section \ref{intsection}, the case where $Y$ is
the circle. In this case Gon\c{c}alves's index does provide a
nontrivial coincidence index which generalizes the fixed point index. 

Thus there are many
examples of maps $f,g:X \to S^1$ for which $\MC(f,g)$, the minimal
number of coincidence points when $f$ and $g$ are changed by homotopy,
is nonzero (this
will occur whenever Gon\c{c}alves's index is nonzero). In particular
when $X$ is also $S^1$, 
it is known that $\MC(f,g)$ is the Nielsen number $N(f,g) = |\deg(f) -
\deg(g)|$, which is easily made nonzero.

\end{document}